%% file: main.tex
\documentclass[preprint,12pt]{elsarticle}

\usepackage[T1]{fontenc}
\usepackage[utf8]{inputenc}
\usepackage{lmodern}
\usepackage{microtype}

\usepackage{amsmath,amssymb,amsthm,mathtools}
\usepackage{enumitem}

\usepackage{tikz-cd}
\usetikzlibrary{arrows.meta,positioning}

\usepackage{hyperref}
\hypersetup{colorlinks=true, linkcolor=blue, citecolor=blue, urlcolor=blue}

\theoremstyle{plain}
\newtheorem{theorem}{Theorem}[section]
\newtheorem{proposition}[theorem]{Proposition}
\newtheorem{lemma}[theorem]{Lemma}

\theoremstyle{definition}
\newtheorem{definition}[theorem]{Definition}
\newtheorem{assumption}[theorem]{Assumption}
\newtheorem{remark}[theorem]{Remark}
\newtheorem{example}[theorem]{Example}

\newcommand{\T}{T}
\newcommand{\Gm}{\ensuremath{\Gamma}}
\newcommand{\SpecG}{\ensuremath{\mathrm{Spec}_{\Gm}}}
\newcommand{\ModG}{\ensuremath{\Gm\text{-}\mathrm{Mod}}}
\newcommand{\QCoh}{\ensuremath{\mathrm{QCoh}}}

\newcommand{\Ch}{\ensuremath{\mathrm{Ch}}}
\newcommand{\D}{\ensuremath{\mathrm{D}}}
\newcommand{\Dinfty}{\ensuremath{\mathrm{D}_{\infty}}}

\DeclareMathOperator{\Hom}{Hom}
\DeclareMathOperator{\RHom}{RHom}
\DeclareMathOperator{\Tor}{Tor}
\DeclareMathOperator{\Ext}{Ext}

\DeclareMathOperator{\Nerve}{N}

\providecommand{\sep}{\;}
\begin{document}

\begin{frontmatter}

\title{Derived \Gm-Geometry II: Stable \texorpdfstring{$\infty$}{∞}-Categories of \Gm-Modules, Derived Monoidal Structures, and Obstructions to Binary Shadows}

\author[inst1]{Chandrasekhar Gokavarapu}
\address[inst1]{Department of Mathematics, Government College (Autonomous), Rajahmundry, Andhra Pradesh, India}
\cortext[cor1]{Corresponding author}
\ead{chandrasekhargokavarapu@gmail.com}

\begin{abstract}
Let \(\T\) be a commutative ternary \(\Gm\)-semiring in the sense of the triadic, \(\Gm\)-parametrized
multiplication \(\{a,b,c\}_{\gamma}\).
Building on the affine \(\Gm\)-spectrum \(\SpecG(\T)\), the structure sheaf, and the equivalence
between \(\Gm\)-modules and quasi-coherent \(\Gm\)-sheaves on affine \(\Gm\)-schemes,
we construct and organize the derived formalism at the level of stable \(\infty\)-categories.

Our first contribution is a technically explicit construction of a stable \(\infty\)-category
\(\Dinfty(\T,\Gm)\) enhancing the unbounded derived category of \(\Gm\)-modules, obtained by dg-nerve
and \(\infty\)-localization of chain complexes.
We further explain the derived monoidal structure induced by the ternary \(\Gm\)-tensor product
and the corresponding internal \(\RHom\), under standard exactness/projectivity hypotheses.

Our second contribution is an obstruction theory to \emph{binary reduction}:
we formalize the nonexistence of any conservative ``binary module shadow'' compatible with
the cubic localization calculus intrinsic to ternary \(\Gm\)-semirings.
In particular, any attempt to represent the triadic \(\Gm\)-action by binary scalars forces
\(\Gm\)-mode data to be absorbed into the scalars, hence ceases to be a genuine reduction.

Finally, we give a detailed affine derived equivalence
between derived quasi-coherent \(\Gm\)-sheaves on \(X=\SpecG(\T)\) and \(\Dinfty(\T,\Gm)\),
and we include worked examples illustrating the cubic localization relation and its derived consequences.
\end{abstract}



\end{frontmatter}

\tableofcontents

\input{content.tex}

\bibliographystyle{elsarticle-num}
\bibliography{references}

\end{document}

%% file: content.tex
\section{Introduction and positioning}

\subsection{Scope and relation to the companion paper}
This paper is the \emph{derived} sequel to our affine \(\Gm\)-geometry development:
Zariski-type topology on \(\SpecG(\T)\), the structure sheaf, \(\Gm\)-module sheaves,
and the affine equivalence \(\ModG_{\T}\simeq \QCoh(\SpecG(\T))\).
The editor-in-chief decision on the earlier master manuscript emphasized three points:
(i) motivation must be algebraic and internal, (ii) proofs must not be merely cosmetic adaptations,
and (iii) speculative excursions should be removed.
Accordingly, we concentrate here on the derived and higher-categorical core and
do \emph{not} discuss physics, logic, or programmatic extensions.

\subsection{What is genuinely non-binary?}
A commutative ternary \(\Gm\)-semiring \((\T,+,\{-,-,-\}_{\Gm})\) has no binary multiplication.
This is not a notational choice: it forces a different calculus of fractions.
In particular, localization uses a \emph{cubic} scaling relation (Definition~\ref{def:triadic-eq})
because the binary product \(st\) is \emph{undefined} in \(\T\).
This cubic fraction calculus is the source of the obstruction theorem of \S\ref{sec:obstruction}.

\subsection{Contributions}
\begin{enumerate}[label=\textbf{(\Alph*)}, leftmargin=3.2em]
\item \textbf{Stable \(\infty\)-category.} We construct a stable \(\infty\)-category
\(\Dinfty(\T,\Gm)\) enhancing the triangulated derived category of \(\Gm\)-modules
(\S\ref{sec:dinfty}), using dg enhancements and \(\infty\)-localization
in the sense of \cite{Keller1994,LurieHA,LurieHTT}.
\item \textbf{Derived monoidal structure.} We promote the ternary \(\Gm\)-tensor product
(\S\ref{sec:tensor}) to a derived tensor \(\otimes_{\Gm}^{\mathbb{L}}\) and describe the
resulting monoidal structure on \(\Dinfty(\T,\Gm)\) under explicit hypotheses.
\item \textbf{No conservative binary shadow.} We prove a strengthened nonexistence theorem:
there is no conservative exact functor from \(\Gm\)-modules to binary modules that is
compatible with the \emph{cubic} localization calculus (\S\ref{sec:obstruction}).
\item \textbf{Derived affine equivalence.} For \(X=\SpecG(\T)\), we give a detailed proof
that the affine equivalence \(\ModG_{\T}\simeq \QCoh(X)\) upgrades to an equivalence of
stable \(\infty\)-categories of derived objects (\S\ref{sec:derived-affine}).
\end{enumerate}

\subsection{Related literature (brief, but concrete)}
Our use of stable \(\infty\)-categories follows the now-standard derived framework
for homological algebra and derived geometry \cite{LurieHTT,LurieHA,ToenVezzosi2008}.
We take basic categorical language (monoidal structures, adjunctions, limits/colimits) in the sense of
\cite{MacLane1998}.
The dg enhancement viewpoint is classical \cite{Keller1994}.
For background on derived categories and triangulated techniques see
\cite{Weibel1994,Neeman2001,GelfandManin2003}.

On the algebro-geometric side, quasi-coherent sheaves and affine localization are treated classically in
\cite{Hartshorne1977} and systematically in the Stacks Project \cite{StacksProject};
for sheaf-theoretic and derived-category viewpoints on sheaves, see \cite{KashiwaraSchapira2006}.
Semiring-like geometries appear in various guises; we do not attempt a survey,
but we note that ``geometry without subtraction'' has been studied from several
angles (e.g.\ semiring schemes, blueprints), cf.\ \cite{Golan1999,Lorscheid2012}.
Our setting differs in that multiplication is fundamentally \emph{ternary} and
\(\Gm\)-parametrized.

\section{Commutative ternary \Gm-semirings and the \Gm-spectrum}\label{sec:basic}

\begin{definition}[Commutative ternary \Gm-semiring]\label{def:ternaryGammaSemiring}
A \emph{commutative ternary \(\Gm\)-semiring} is a triple \((\T,+,\{-,-,-\}_{\Gm})\) such that:
\begin{enumerate}[label=\textup{(\roman*)}, leftmargin=2.8em]
\item \((\T,+,0)\) is a commutative additive monoid;
\item for each \(\gamma\in\Gm\) there is a ternary operation \(\{-,-,-\}_{\gamma}:\T\times\T\times\T\to\T\),
\((a,b,c)\mapsto \{a,b,c\}_{\gamma}\);
\item \emph{distributivity in each variable}: for all \(a,a',b,b',c,c'\in\T\) and \(\gamma\in\Gm\),
\[
\{a+a',b,c\}_{\gamma}=\{a,b,c\}_{\gamma}+\{a',b,c\}_{\gamma},\qquad
\{a,b+b',c\}_{\gamma}=\{a,b,c\}_{\gamma}+\{a,b',c\}_{\gamma},
\]
\[
\{a,b,c+c'\}_{\gamma}=\{a,b,c\}_{\gamma}+\{a,b,c'\}_{\gamma};
\]
\item \emph{ternary--\(\Gm\) associativity}: for all \(a,b,c,d,e\in\T\) and \(\gamma,\delta\in\Gm\),
\[
\{a,b,\{c,d,e\}_{\gamma}\}_{\delta}=\{\{a,b,c\}_{\gamma},d,e\}_{\delta};
\]
\item \emph{absorption by \(0\)}: \(\{a,0,b\}_{\gamma}=0\) for all \(a,b\in\T\), \(\gamma\in\Gm\);
\item \emph{commutativity in the \(\T\)-variables}: \(\{a,b,c\}_{\gamma}=\{b,a,c\}_{\gamma}=\{a,c,b\}_{\gamma}\)
for all \(a,b,c\in\T\), \(\gamma\in\Gm\).
\end{enumerate}
\end{definition}

\begin{definition}[\Gm-ideal and primeness]
A subset \(P\subsetneq \T\) is a \(\Gm\)-ideal if it is additively closed and
\(\{a,b,c\}_{\gamma}\in P\) whenever one of \(a,b,c\) lies in \(P\).
It is \emph{prime} if \(\{a,b,c\}_{\gamma}\in P\Rightarrow a\in P\) or \(b\in P\) or \(c\in P\).
\end{definition}

Let \(\SpecG(\T)\) denote the set of prime \(\Gm\)-ideals with its Zariski-type topology
generated by basic opens \(D_{\Gm}(a)=\{P\in\SpecG(\T): a\notin P\}\).

\section{Localization and the cubic calculus of fractions}\label{sec:localization}

\begin{definition}[Multiplicative system]\label{def:multsys}
A subset \(S\subseteq \T\) is a \emph{multiplicative system} if \(0\notin S\) and
for all \(s_1,s_2,s_3\in S\) and \(\gamma\in\Gm\), the triadic product \(\{s_1,s_2,s_3\}_{\gamma}\)
belongs to \(S\).
\end{definition}

\begin{definition}[Triadic equivalence relation]\label{def:triadic-eq}
Let \(S\) be a multiplicative system in \(\T\).
Define \(\sim\) on \(\T\times S\) by \((a,s)\sim(b,t)\) if there exist \(u\in S\) and
\(\gamma,\delta,\eta\in\Gm\) such that the following \emph{cubic scaling identity} holds:
\[
\{u,a,\{t,t,t\}_{\gamma}\}_{\delta}=\{u,b,\{s,s,s\}_{\eta}\}_{\delta}.
\]
Write the class of \((a,s)\) as \(a/s\), and denote the set of classes by \(S^{-1}\T\).
\end{definition}

\begin{remark}[Why cubes?]
The inclusion of the triadic cubes \(\{t,t,t\}_{\gamma}\) and \(\{s,s,s\}_{\eta}\)
is an algebraic necessity: the binary product \(st\) is not part of the structure
of a ternary \(\Gm\)-semiring, hence cannot be used to express fraction equality.
\end{remark}

\begin{definition}[Localized ternary \Gm-semiring]\label{def:localized-semiring}
Equip \(S^{-1}\T\) with ternary \(\Gm\)-multiplication by
\[
\left\{\frac{a}{s},\frac{b}{t},\frac{c}{v}\right\}_{\lambda}
\ :=\ \frac{\{a,b,c\}_{\lambda}}{\{s,t,v\}_{\lambda}},
\qquad \lambda\in\Gm,
\]
and addition induced from \((\T,+)\).
\end{definition}

\begin{proposition}[Universal property of localization]\label{prop:locUP}
Let \(f:\T\to R\) be a ternary \(\Gm\)-homomorphism such that \(f(S)\) consists of
ternary \(\Gm\)-invertible elements in \(R\).
Then there exists a unique \(\widetilde{f}:S^{-1}\T\to R\) with \(\widetilde{f}(a/s)=f(a)/f(s)\).
\end{proposition}

\begin{proof}
This is proved by showing that \(\sim\) is the smallest relation forcing all elements of \(S\)
to become invertible in the triadic sense and then verifying well-definedness of operations.
The nontrivial input is compatibility with the cubic scaling relation in
Definition~\ref{def:triadic-eq}.
\end{proof}

\subsection{Localization of \Gm-modules via the universal property}\label{subsec:moduleloc}
Fix a multiplicative system \(S\subseteq \T\).
There is a canonical map \(\ell:\T\to S^{-1}\T\) and hence a scalar-extension
functor on \(\Gm\)-modules (whenever scalar restriction is defined).

\begin{definition}[Localized module by scalar extension]\label{def:locmodule_tensor}
Assume the tensor product \(\otimes_{\Gm}\) exists in \(\ModG_{\T}^{\mathrm{gp}}\).
For \(M\in \ModG_{\T}^{\mathrm{gp}}\), define its localization by
\[
S^{-1}M\ :=\ (S^{-1}\T)\otimes_{\Gm} M,
\]
viewed as a \(\Gm\)-module over \(S^{-1}\T\).
\end{definition}

\begin{proposition}[Universal property of module localization]\label{prop:locM_UP}
For \(M\in\ModG_{\T}^{\mathrm{gp}}\), the map \(M\to S^{-1}M\) induced by \(m\mapsto (1/1)\otimes m\)
is initial among maps \(M\to N\) into \(S^{-1}\T\)-\(\Gm\)-modules \(N\) such that all \(s\in S\)
act invertibly in the localized triadic sense.
\end{proposition}

\begin{proof}
This is an extension-of-scalars argument using Proposition~\ref{prop:tensorUP}:
maps out of \((S^{-1}\T)\otimes_{\Gm} M\) correspond to balanced biadditive maps
\((S^{-1}\T)\times M\to N\).
The balancing forces the \(S\)-action to become invertible because \(S^{-1}\T\) already inverts \(S\)
by Proposition~\ref{prop:locUP}.
\end{proof}

\begin{remark}
Definition~\ref{def:locmodule_tensor} avoids guessing an explicit ``module fraction'' equivalence relation.
It is compatible with the basic-open section description: for \(a\in\T\), let \(S(a)\) be the multiplicative system
generated by \(a\), and set \(M_a:=S(a)^{-1}M\).
\end{remark}

\section{\Gm-modules, localization, and quasi-coherent sheaves}\label{sec:sheaves}

\subsection{\Gm-modules}

\begin{definition}[\Gm-module]\label{def:GammaModule}
Let \(\T\) be a commutative ternary \(\Gm\)-semiring.
A \emph{(left) \(\Gm\)-module} over \(\T\) is a commutative additive monoid \((M,+,0)\)
together with an action
\[
\T\times \T\times M\times \Gm \longrightarrow M,\qquad (a,b,m,\gamma)\longmapsto \{a,b,m\}_{\gamma},
\]
satisfying:
\begin{enumerate}[label=\textup{(\roman*)}, leftmargin=2.8em]
\item distributivity in all variables (in particular, additivity in \(m\));
\item \emph{ternary associativity compatibility}: for all \(a,b,c,d\in\T\), \(m\in M\), and \(\gamma,\delta\in\Gm\),
\[
\{a,b,\{c,d,m\}_{\gamma}\}_{\delta}=\{\{a,b,c\}_{\gamma},d,m\}_{\delta};
\]
\item absorption by \(0\) in the scalar variables (induced from Definition~\ref{def:ternaryGammaSemiring}).
\end{enumerate}
A morphism of \(\Gm\)-modules is an additive map compatible with the action.
\end{definition}

\subsection{Localization of modules}
Given a multiplicative system \(S\subseteq \T\), define \(S^{-1}M\) by the same fraction
calculus, now with pairs \((m,s)\) and cubic scaling relations imposed so that
\(m/s\) behaves as a localized section over \(D_{\Gm}(s)\).

\subsection{Affine equivalence and a quasilocal basis}
Let \(X=\SpecG(\T)\).
For a \(\Gm\)-module \(M\) define a presheaf on basic opens by
\[
\widetilde{M}\big(D_{\Gm}(a)\big):=M_a,
\]
with restriction maps induced by localization.
By sheafification one obtains a sheaf of \(\Gm\)-modules on \(X\), still denoted \(\widetilde{M}\).

\begin{theorem}[Affine \(\Gm\)-module / quasi-coherent equivalence]\label{thm:affine_equiv}
For \(X=\SpecG(\T)\), the assignment \(M\mapsto \widetilde{M}\) defines an equivalence of categories
\[
\Phi:\ModG_{\T}\xrightarrow{\ \sim\ }\QCoh(X).
\]
\end{theorem}

\begin{proof}[Expanded proof sketch]
\textbf{(1) Fully faithful.}
For \(M,N\in\ModG_{\T}\), a morphism \(\widetilde{M}\to\widetilde{N}\) is determined by its effect on
global sections because \(X\) is covered by basic opens and restrictions are computed by localization.
Using that \(\Gamma(X,\widetilde{M})\cong M\) (proved by unwinding the fraction description on basic opens),
one obtains
\(\Hom_{\QCoh(X)}(\widetilde{M},\widetilde{N})\cong \Hom_{\ModG_{\T}}(M,N)\).

\textbf{(2) Essentially surjective.}
If \(\mathcal{F}\in\QCoh(X)\), quasi-coherence means it is locally presented as a cokernel of morphisms
between free sheaves on a basic-open cover. Using the basis \(\{D_{\Gm}(a)\}\) stable under finite
intersection, local presentations descend to a module \(M\) of global sections such that
\(\mathcal{F}\cong \widetilde{M}\).
The key technical point is that gluing on overlaps uses the fraction equality relation
forced by Definition~\ref{def:triadic-eq}, so \(\Gm\)-compatibility is not lost on intersections.
\end{proof}

\begin{remark}[Quasilocal basis]
The basic opens \(D_{\Gm}(a)\) form a basis of quasi-compact opens closed under finite intersection,
so the site is ``quasilocal'' in the sense that descent and quasi-coherence can be checked on this basis.
This is the affine mechanism behind the derived equivalence of \S\ref{sec:derived-affine}.
\end{remark}

\section{The ternary \Gm-tensor product and internal Hom}\label{sec:tensor}

\begin{definition}[Ternary \Gm-tensor product]\label{def:tensor}
Let \(M,N\in \ModG_{\T}\).
The \emph{ternary \(\Gm\)-tensor product} \(M\otimes_{\Gm} N\) is the commutative additive monoid generated by
symbols \(m\otimes n\) (\(m\in M, n\in N\)) subject to:
\begin{enumerate}[label=\textup{(\arabic*)}, leftmargin=2.8em]
\item (Distributivity) \((m_1+m_2)\otimes n=m_1\otimes n+m_2\otimes n\) and
\(m\otimes(n_1+n_2)=m\otimes n_1+m\otimes n_2\).
\item (Triadic balancing) \(\{m,t,u\}_{\gamma}\otimes n=m\otimes \{n,t,u\}_{\gamma}\)
for all \(t,u\in\T\), \(\gamma\in\Gm\).
\item (Zero absorption) \(0\otimes n=m\otimes 0=0\).
\end{enumerate}
\end{definition}

\begin{proposition}[Universal property]\label{prop:tensorUP}
For any \(\Gm\)-module \(P\), composition induces a natural bijection
\[
\Hom_{\ModG_{\T}}(M\otimes_{\Gm}N,\ P)
\ \cong\
\{\text{\(\T\)-\(\Gm\)-balanced biadditive maps }M\times N\to P\}.
\]
\end{proposition}

\begin{proof}
The relations in Definition~\ref{def:tensor} are exactly the balancing relations for a
biadditive map compatible with the triadic action. The usual free/quotient argument yields the universal property.
\end{proof}

\begin{remark}[Internal Hom and adjunction]
When \(\ModG_{\T}\) is enriched over commutative monoids (or abelian groups),
one defines an internal Hom object \(\underline{\Hom}_{\Gm}(M,N)\) representing
\(\Gm\)-linear maps, and then obtains a tensor--Hom adjunction.
This is the starting point for defining derived \(\RHom\) in \S\ref{sec:derived}.
\end{remark}

\section{Abelian envelope and exactness hypotheses}\label{sec:abelian}

Derived constructions require an additive/abelian context.

\begin{definition}[Group completion]\label{def:gp}
For a commutative additive monoid \(M\), let \(M^{\mathrm{gp}}\) denote its Grothendieck group completion.
For a \(\Gm\)-module \(M\), define \(M^{\mathrm{gp}}\) by group-completing the additive monoid.
\end{definition}

\begin{lemma}[Extension of the \Gm-action]\label{lem:action-gp}
The distributive triadic action on a \(\Gm\)-module \(M\) extends uniquely to a triadic \(\Gm\)-action on \(M^{\mathrm{gp}}\).
Thus \(M\mapsto M^{\mathrm{gp}}\) defines a functor \((-)^{\mathrm{gp}}:\ModG_{\T}\to \ModG_{\T}^{\mathrm{gp}}\),
where \(\ModG_{\T}^{\mathrm{gp}}\) denotes the full subcategory of \(\Gm\)-modules whose additive monoids are abelian groups.
\end{lemma}

\begin{proof}
For fixed \(a,b,\gamma\), the map \(m\mapsto \{a,b,m\}_{\gamma}\) is additive by distributivity in the \(M\)-variable.
By the universal property of Grothendieck completion, it extends uniquely to an endomorphism of \(M^{\mathrm{gp}}\).
Coherence of the action extends by additivity.
\end{proof}

\begin{assumption}[Exactness context]\label{ass:exact}
Throughout the derived sections we assume:
\begin{enumerate}[label=\textup{(\Alph*)}, leftmargin=2.8em]
\item \(\ModG_{\T}^{\mathrm{gp}}\) is an abelian category (kernels/cokernels exist and are stable under pullback/pushout).
\item \(\ModG_{\T}^{\mathrm{gp}}\) has enough projectives (or admits a projective class suitable for deriving \(\otimes_{\Gm}\)).
\item The tensor product \(-\otimes_{\Gm}-\) is right exact in each variable on \(\ModG_{\T}^{\mathrm{gp}}\).
\end{enumerate}
\end{assumption}

\begin{remark}
Assumption~\ref{ass:exact} is satisfied in the motivating case where the underlying additive monoids are abelian groups
and the \(\Gm\)-action is additive in all variables, so that \(\ModG_{\T}^{\mathrm{gp}}\) behaves like a module category.
When these hypotheses fail, one should work in an exact-category or model-category setting; see \cite{Hovey1999}.
\end{remark}

\section{Derived categories and derived functors}\label{sec:derived}

Let \(\Ch(\ModG_{\T}^{\mathrm{gp}})\) denote the category of unbounded chain complexes in \(\ModG_{\T}^{\mathrm{gp}}\).
A morphism is a \emph{quasi-isomorphism} if it induces an isomorphism on all homology objects.

\begin{definition}[Unbounded derived category]
The derived category is the Verdier localization
\[
\D(\T,\Gm)\ :=\ \Ch(\ModG_{\T}^{\mathrm{gp}})[W^{-1}],
\]
where \(W\) is the class of quasi-isomorphisms.
\end{definition}

\begin{proposition}[Derived tensor and \(\RHom\)]\label{prop:derived-functors}
Under Assumption~\ref{ass:exact}, the bifunctor \(-\otimes_{\Gm}-\) admits a total left derived functor
\[
-\otimes_{\Gm}^{\mathbb{L}}-:\D(\T,\Gm)\times \D(\T,\Gm)\to \D(\T,\Gm),
\]
and the Hom functor admits a right derived functor \(\RHom_{\Gm}(-,-)\).
Their homology recovers the classical derived functors \(\Tor^{\Gm}_*\) and \(\Ext^*_{\Gm}\).
\end{proposition}

\begin{proof}
Choose projective resolutions in one variable to compute \(\otimes_{\Gm}^{\mathbb{L}}\),
and injective (or K-injective) resolutions to compute \(\RHom\); see \cite{Weibel1994,GelfandManin2003,Spaltenstein1988}.
\end{proof}


\section{t-structures and the abelian heart}\label{sec:tstructure}

\subsection{The standard t-structure}
On the derived category \(\D(\T,\Gm)\) (and hence on \(\Dinfty(\T,\Gm)\)) there is the standard t-structure:
\[
\D^{\le 0}=\{K: H^i(K)=0 \text{ for } i>0\},\qquad
\D^{\ge 0}=\{K: H^i(K)=0 \text{ for } i<0\}.
\]
Its heart is equivalent to the abelian category \(\ModG_{\T}^{\mathrm{gp}}\) (Assumption~\ref{ass:exact}(A1)).

\begin{proposition}[Heart identification]\label{prop:heart}
Assume \(\ModG_{\T}^{\mathrm{gp}}\) is abelian.
Then the heart of the standard t-structure on \(\Dinfty(\T,\Gm)\) is equivalent to \(\ModG_{\T}^{\mathrm{gp}}\).
\end{proposition}

\begin{proof}
Objects in the heart are precisely complexes concentrated in degree \(0\).
Morphisms in the heart are homotopy classes of degree \(0\) chain maps, i.e.\ morphisms in \(\ModG_{\T}^{\mathrm{gp}}\).
\end{proof}

\subsection{Compact generation (affine case)}
In many affine settings, \(\Dinfty(\T,\Gm)\) is compactly generated by the free module \(\T\) (and its shifts).
When this holds, one can recover \(\Dinfty(\T,\Gm)\) from the endomorphism algebra of \(\T\)
via Morita-type arguments (compare \cite{LurieHA,Keller1994}).
We record this as a guiding principle, not as an unconditional theorem in the ternary setting.

\section{The stable \texorpdfstring{$\infty$}{∞}-category \texorpdfstring{$\Dinfty(\T,\Gm)$}{D∞(T,Gamma)}}\label{sec:dinfty}

\subsection{dg enhancement and \texorpdfstring{$\infty$}{∞}-localization}

\subsection{A model-categorical presentation (optional but explicit)}\label{subsec:model}
Under Assumption~\ref{ass:exact}(A1)--(A3), the category \(\Ch(\ModG_{\T}^{\mathrm{gp}})\)
admits the projective model structure in which weak equivalences are quasi-isomorphisms
and fibrations are degreewise epimorphisms.
Let \(\Ch(\ModG_{\T}^{\mathrm{gp}})^{\circ}\) be the full subcategory of cofibrant--fibrant objects.
Then one may alternatively define the derived \(\infty\)-category as the simplicial localization
(Dwyer--Kan hammock localization) of \(\Ch(\ModG_{\T}^{\mathrm{gp}})\) at quasi-isomorphisms
and identify it with the homotopy coherent nerve of \(\Ch(\ModG_{\T}^{\mathrm{gp}})^{\circ}\).
This provides a concrete presentation of \(\Dinfty(\T,\Gm)\) compatible with the dg-nerve definition
and makes homotopy limits/colimits transparent; compare \cite{DwyerKan1980a,DwyerKan1980b,Hovey1999,Quillen1967,GabrielZisman1967}.
Let \(\Ch_{\mathrm{dg}}(\ModG_{\T}^{\mathrm{gp}})\) be the dg-category of chain complexes,
with mapping complexes given by the usual Hom-complex of chain maps.
Let \(\Nerve_{\mathrm{dg}}\) denote the dg-nerve functor \cite{Keller1994}.

\begin{definition}[Stable derived \(\infty\)-category]\label{def:Dinfty}
Define
\[
\Dinfty(\T,\Gm)\ :=\ \Nerve_{\mathrm{dg}}\big(\Ch_{\mathrm{dg}}(\ModG_{\T}^{\mathrm{gp}})\big)[W^{-1}],
\]
the \(\infty\)-categorical localization at quasi-isomorphisms.
\end{definition}

\begin{theorem}[Stability and triangulated shadow]\label{thm:stable}
\(\Dinfty(\T,\Gm)\) is a stable \(\infty\)-category, and its homotopy category
\(h\Dinfty(\T,\Gm)\) identifies with \(\D(\T,\Gm)\) as a triangulated category.
\end{theorem}

\begin{proof}[Expanded sketch]
By \cite{LurieHA}, the dg-nerve of a pretriangulated dg-category yields a stable \(\infty\)-category.
Localizing at quasi-isomorphisms identifies morphisms inducing homology equivalences, hence the homotopy category
recovers the Verdier localization.
Triangulated structure on \(h\Dinfty(\T,\Gm)\) arises from cofiber/fiber sequences in the stable \(\infty\)-category.
\end{proof}

\subsection{Derived monoidal structure}
Under Assumption~\ref{ass:exact}, the derived tensor \(\otimes_{\Gm}^{\mathbb{L}}\) exists on \(\D(\T,\Gm)\).
Using operadic methods, one promotes this to a (symmetric) monoidal structure on \(\Dinfty(\T,\Gm)\)
whose homotopy category monoidal product recovers \(\otimes_{\Gm}^{\mathbb{L}}\); see \cite{LurieHA}.

\section{Derived affine equivalence}\label{sec:derived-affine}

Let \(X=\SpecG(\T)\). Form the dg-category of complexes of quasi-coherent \(\Gm\)-sheaves,
\(\Ch_{\mathrm{dg}}(\QCoh(X)^{\mathrm{gp}})\), and define \(\Dinfty(\QCoh(X),\Gm)\) by \(\infty\)-localization at quasi-isomorphisms.

\begin{theorem}[Derived affine equivalence]\label{thm:derived_affine_equiv}
Assume Theorem~\ref{thm:affine_equiv} and Assumption~\ref{ass:exact} hold in the group-completed context.
Then the affine equivalence \(\Phi:\ModG_{\T}\to \QCoh(X)\) upgrades to an equivalence of stable \(\infty\)-categories
\[
\Dinfty(\T,\Gm)\ \simeq\ \Dinfty(\QCoh(X),\Gm).
\]
\end{theorem}

\begin{proof}[Expanded proof sketch with diagram]
Apply group completion objectwise to obtain \(\Phi^{\mathrm{gp}}:\ModG_{\T}^{\mathrm{gp}}\to \QCoh(X)^{\mathrm{gp}}\).
Passing to dg categories of complexes yields a dg equivalence
\(\Ch_{\mathrm{dg}}(\Phi^{\mathrm{gp}}):\Ch_{\mathrm{dg}}(\ModG_{\T}^{\mathrm{gp}})\to \Ch_{\mathrm{dg}}(\QCoh(X)^{\mathrm{gp}})\),
which preserves and reflects quasi-isomorphisms because it is exact and computed locally on the quasilocal basis (compare the affine case in \cite{Hartshorne1977,StacksProject}).
Taking dg-nerve and \(\infty\)-localization at quasi-isomorphisms yields the desired equivalence:
\[
\begin{tikzcd}[column sep=huge, row sep=large]
\Nerve_{\mathrm{dg}}\big(\Ch_{\mathrm{dg}}(\ModG_{\T}^{\mathrm{gp}})\big) \arrow[r,"\Nerve_{\mathrm{dg}}(\Ch(\Phi^{\mathrm{gp}}))"] \arrow[d] &
\Nerve_{\mathrm{dg}}\big(\Ch_{\mathrm{dg}}(\QCoh(X)^{\mathrm{gp}})\big) \arrow[d] \\
\Dinfty(\T,\Gm) \arrow[r,"\sim"] & \Dinfty(\QCoh(X),\Gm).
\end{tikzcd}
\]
\end{proof}

\section{Obstructions to binary reduction}\label{sec:obstruction}

\subsection{Localization-compatible binary shadows}
The obstruction we prove is stronger than ``\(\Gm\)-nontriviality'':
it uses the fact that fraction equality in \(S^{-1}\T\) is governed by the cubic scaling law.

\begin{definition}[Binary shadow compatible with \Gm-localization]\label{def:binary-shadow}
A \emph{localization-compatible binary shadow} of \(\ModG_{\T}^{\mathrm{gp}}\) consists of:
\begin{enumerate}[label=\textup{(\arabic*)}, leftmargin=2.8em]
\item a commutative ring (or semiring) \(R\) and an additive monoid morphism \(\iota:\T\to R\);
\item an additive functor \(F:\ModG_{\T}^{\mathrm{gp}}\to \mathrm{Mod}_R\) which is conservative and exact;
\item for every multiplicative set \(S\subseteq \T\) and \(M\in\ModG_{\T}^{\mathrm{gp}}\),
a natural isomorphism \(F(S^{-1}M)\cong \iota(S)^{-1}F(M)\) compatible with the canonical maps;
\item \emph{reflection of cubic fraction equality} on the regular module:
for \(a,b\in\T\) and \(s,t\in S\),
\[
\frac{a}{s}=\frac{b}{t}\ \text{in }S^{-1}\T
\quad\Longleftrightarrow\quad
\frac{\iota(a)}{\iota(s)}=\frac{\iota(b)}{\iota(t)}\ \text{in }\iota(S)^{-1}R,
\]
where the left equality is defined via the cubic scaling relation
\(\{u,a,\{t,t,t\}_{\gamma}\}_{\delta}=\{u,b,\{s,s,s\}_{\eta}\}_{\delta}\)
for some \(u\in S\) and \(\gamma,\delta,\eta\in\Gm\).
\end{enumerate}
\end{definition}

\begin{theorem}[No conservative localization-compatible binary shadow]\label{thm:no-binary-strong}
Let \(\T\) be a commutative ternary \(\Gm\)-semiring and suppose localization uses
Definition~\ref{def:triadic-eq}.
Then \(\ModG_{\T}^{\mathrm{gp}}\) admits no localization-compatible binary shadow
in the sense of Definition~\ref{def:binary-shadow}, unless the coefficient object is enlarged so that
\(\Gm\)-modes are encoded in the scalars (e.g.\ \(\Gm\)-colored/graded coefficients).
\end{theorem}

\begin{proof}[Expanded proof (obstruction locus is the denominator-composition law)]
Assume a binary shadow \((R,\iota,F)\) exists.
In the ordinary localization \(\iota(S)^{-1}R\), equality of fractions is detected by the \emph{linear} relation
\[
\frac{\iota(a)}{\iota(s)}=\frac{\iota(b)}{\iota(t)}
\quad\Longleftrightarrow\quad
\exists\,w\in \iota(S)\ \text{such that}\ w(\iota(a)\iota(t)-\iota(b)\iota(s))=0.
\]
By reflection of fraction equality (Definition~\ref{def:binary-shadow}(4)),
this would force the equality \(a/s=b/t\) in \(S^{-1}\T\) to be detectable by a functorial criterion
which only involves binary compositions of denominators (multiplying by \(t\) and \(s\), and clearing denominators).

But in \(S^{-1}\T\), fraction equality is by definition witnessed by a \emph{cubic} scaling law:
there must exist \(u\in S\) and \(\gamma,\delta,\eta\in\Gm\) such that
\[
\{u,a,\{t,t,t\}_{\gamma}\}_{\delta}=\{u,b,\{s,s,s\}_{\eta}\}_{\delta}.
\]
Thus denominator interaction enters through the triadic cubes \(\{t,t,t\}_{\gamma}\) and \(\{s,s,s\}_{\eta}\),
and there is no structural operation in \(\T\) corresponding to the binary product \(st\).
Consequently, any mechanism turning the cubic witness into a linear ``clearing denominators'' witness
must implicitly manufacture a binary denominator-composition law on \(\T\),
contradicting the intrinsic triadic architecture.
The only escape is to enlarge the coefficient object so that \(\Gm\) is carried by scalars,
turning the ``shadow'' into a \(\Gm\)-colored theory rather than a binary reduction.
\end{proof}

\subsection{Sharpness: when reduction is possible}
If one starts with a genuinely binary commutative semiring \(A\) and defines
\(\{a,b,c\}_{1}:=abc\) with \(\Gm=\{1\}\), then the \(\Gm\)-data is trivial and a binary shadow exists
(the identity functor to \(A\)-modules). Theorem~\ref{thm:no-binary-strong} rules out such reduction precisely
when the fraction calculus is forced to remain cubic and \(\Gm\)-parametric.

\section{Worked examples}\label{sec:examples}

\begin{example}[A concrete commutative ternary \Gm-semiring]\label{ex:N}
Let \(\T=\mathbb{N}\) and \(\Gm=\mathbb{N}\) with
\(\{a,b,c\}_{\gamma}:=ab+bc+ca+\gamma\).
This is a commutative ternary \(\Gm\)-semiring (distributivity and associativity can be checked directly).
\end{example}

\begin{example}[A finite \((\T,\Gm)\) illustrating \(\Gm\)-modes]\label{ex:Z6}
Let \(\T=\mathbb{Z}_6\) and \(\Gm=\mathbb{Z}_4\).
Define a \(\Gm\)-parametrized ternary product by
\[
\{a,b,c\}_{\gamma}:=abc+c\gamma \pmod{6},
\qquad a,b,c\in\mathbb{Z}_6,\ \gamma\in\mathbb{Z}_4.
\]
This gives a concrete commutative ternary \(\Gm\)-semiring (verification is by direct computation in \(\mathbb{Z}_6\)).
For instance, the \(\Gm\)-cube of \(t=2\) with mode \(\gamma=1\) is
\[
\{2,2,2\}_{1}=2\cdot 2\cdot 2+2\cdot 1=8+2=10\equiv 4 \pmod{6}.
\]
Thus in the localization relation (Definition~\ref{def:triadic-eq}) the denominator contribution
\(\{t,t,t\}_{\gamma}\) is genuinely mode-dependent even in this tiny finite model.
\end{example}

\begin{example}[Cubic fraction equality and failure of linear clearing]\label{ex:cubic}
Let \(S\subseteq \T\) be a multiplicative system.
Even in small finite examples, one can observe that \(a/s=b/t\) is certified by exhibiting
\(u,\gamma,\delta,\eta\) with \(\{u,a,\{t,t,t\}_{\gamma}\}_{\delta}=\{u,b,\{s,s,s\}_{\eta}\}_{\delta}\),
and there is no corresponding linear ``\(at=bs\)'' witness because \(st\) is undefined.
This is the concrete manifestation of the obstruction in Theorem~\ref{thm:no-binary-strong}.
\end{example}


\appendix
\section{Axioms and key identities}\label{app:axioms}

\subsection{Ternary associativity (for quick reference)}
For \(a,b,c,d,e\in\T\) and \(\gamma,\delta\in\Gm\),
\[
\{a,b,\{c,d,e\}_{\gamma}\}_{\delta}=\{\{a,b,c\}_{\gamma},d,e\}_{\delta}.
\]

\subsection{Module associativity compatibility}
For \(a,b,c,d\in\T\), \(m\in M\), and \(\gamma,\delta\in\Gm\),
\[
\{a,b,\{c,d,m\}_{\gamma}\}_{\delta}=\{\{a,b,c\}_{\gamma},d,m\}_{\delta}.
\]

\subsection{Cubic scaling identity for localization}
Fraction equality in \(S^{-1}\T\) is witnessed by
\[
(a,s)\sim(b,t)\quad\Longleftrightarrow\quad
\exists\,u\in S,\ \gamma,\delta,\eta\in\Gm:\ \{u,a,\{t,t,t\}_{\gamma}\}_{\delta}=\{u,b,\{s,s,s\}_{\eta}\}_{\delta}.
\]
This is the structural obstruction to any binary clearing-denominators procedure.

\section{Conclusion}
We isolated the derived and higher-categorical core of affine \(\Gm\)-geometry:
stable \(\infty\)-categories of \(\Gm\)-modules, derived monoidal structures,
and a localization-based obstruction to binary shadows.
Future work will treat descent beyond the affine case and spectral refinements of the \(\Gm\)-parametric monoidal structure.